\documentclass{amsart}%
\usepackage{amssymb, amsmath, amsfonts, graphicx}%
\usepackage{marginnote}
\usepackage{pgfplots}
\usepackage{filecontents}
\usepackage{tikz}
\usepackage{float}
\usetikzlibrary{arrows,automata}
\usetikzlibrary{graphs,graphs.standard,quotes}

\newtheorem{theorem}{Theorem}

\newtheorem{corollary}[theorem]{Corollary}

\newtheorem{definition}[theorem]{Definition}

\newtheorem{lemma}[theorem]{Lemma}

\newtheorem{proposition}[theorem]{Proposition}
\newtheorem{question}[theorem]{Question}

\pgfplotsset{compat=1.14}

\begin{document}


\title[Ramsey theory and Big Data]{Using Ramsey theory to measure unavoidable spurious correlations in Big Data}

\author[M. Pawliuk]{Micheal Pawliuk}
\address{Department of Mathematics and Statistics\\ University of Calgary\\ Calgary, Canada}
\email{mpawliuk@ucalgary.ca}

\author[M. Waddell]{Michael Alexander Waddell}
\address{Department of Engineering\\ Columbia University \\ New York, USA}
\email{maw2240@columbia.edu}

\subjclass[2000]{Primary: 62-07, 05-D10}
\keywords{Statistics, Data Analysis, Ramsey theory, Graph theory, Transitivity, Hamming distance}

\begin{abstract}
Given a dataset we quantify how many patterns must always exist in the dataset. Formally this is done through the lens of Ramsey theory of graphs, and a quantitative bound known as Goodman's theorem. Combining statistical tools with Ramsey theory of graphs gives a nuanced understanding of how far away a dataset is from random, and what qualifies as a meaningful pattern.

This method is applied to a dataset of repeated voters in the 1984 US congress, to quantify how homogeneous a subset of congressional voters is. We also measure how transitive a subset of voters is. Statistical Ramsey theory is also used with global economic trading data to provide evidence that global markets are quite transitive. 
\end{abstract}

\date{\today}

\maketitle

\tableofcontents

\section{Introduction}
\label{sec:Intro}

In the realm of data science, the conventional wisdom is that ``more data is always better", but is this  the case? As a dataset $D$ becomes larger, Ramsey theory describes the mathematical conditions by which disorder becomes impossible. The impossibility of disorder is analogous to the existence of unavoidable and spurious correlations in large datasets. This makes understanding and measuring the extent of these spurious correlations essential in any attempt to glean meaningful information about $D$. In 2016 \cite{calude2017deluge}, Calude asked the question, how can Ramsey theory be used to understand spurious and unavoidable correlations in data science?

For example, the pigeonhole principle is an extreme, basic version of the Ramsey statement, ``if a given person wears 8 different shirts in a given week, then there must have been a day where they wore at least 2 shirts." Here the dataset is the collection of shirts, with each shirt assigned a day. The unavoidable spurious correlation is that (at least) two shirts are assigned to the same day. In this case, there is no meaningful conclusion we can draw, despite the natural human desire to attribute meaning to a pattern that is observed.

However, we might try to draw meaningful conclusions if we identify a day where the person wore 3 shirts on the same day, or multiple days where they wore multiple shirts, because the pigeonhole principle on its own cannot guarantee these beyond the base requirement that there is a single day where two shirts must be worn. 

Goodman's formula \cite{goodman1959} provides a way to calculate the required number of certain relationships in a relational database. We use Goodman's formula to quantify how many correlations must be observed to ensure that some of the correlations are not spurious. Put another way, we use Goodman's formula to test the null hypothesis $H_0$ that a graph representing the relationships in a dataset is random. 

In section \ref{sec:defns} we present the relevant definitions and mathematical framework. In section \ref{sec:Ramsey} we introduce the needed Ramsey technology of Goodman's formula. In section \ref{sec:Models} we apply this to two real life models: (1) similarity of voting records are for the members of the 1984 US congress, and (2) economic trading data between countries. In section \ref{sec:trans} we give an application of Goodman's formula to measuring the transitivity of a graph. Finally, in section \ref{sec:concl} we discuss further directions for research.

\section{Mathematical framework}
\label{sec:defns}

Our main model is a graph $G$, which is a collection of data points $V$, called the vertices, and a collection of connected (unordered) pairs of vertices $E$, called the edges, such that $G=(V,E)$. An edge $a$ between vertices $v_1$ and $v_2$ represents that $v_1$ and $v_2$ are related (in an abstract sense). This edge relationship will be intrinsic to each dataset and what it is trying to measure. For example, if the vertices are points in a metric space we might assign an edge when the distance between two points is $\leq 1$. Another example is when the vertices are people in a room, and we put an edge between two people if they are both friends.

We insist that a vertex cannot be related to itself (a so-called loop) and that it can be described as an adjacency matrix by explicitly listing out which vertices have an edge between them:

\begin{equation}
{A(G_6) = \bordermatrix{~
	 &v_1 & v_2 & v_3&v_4&v_5&v_6\cr
 v_1&0 & a & b&c&d&e\cr
 v_2&a & 0 & f&g&h&i\cr
 v_3&b & f & 0&j&k&l\cr
 v_4&c & g & j&0&m&n\cr
 v_5&d & h & k&m&0&o \cr
 v_6&e & i & l&n&o&0
 }}.
\end{equation}

A matrix $A = [a_{ij}]_{1 \leq i,j \leq N}$ is an \textbf{adjacency matrix} if it is symmetric with entries of $0, 1$ with $0$s along the diagonal. An adjacency matrix can be thought of as a graph on vertices $\{1, \dots, N\}$ where there is an edge between $i$ and $j$ iff $a_{ij} = 1$. This perspective is useful for the following reason:

\begin{lemma} Let $A$ be an $N\times N$ adjacency matrix, and $k \geq 1$. In the matrix $A^k = A \cdot A \cdot \ldots \cdot A$ ($k$-times) the $ij^{th}$ entry is the number of paths in $A$ from $i$ to $j$ of length exactly $k$.
\end{lemma}

In other words, if the first power $(k=1)$ of the adjacency matrix $A$ represents an edge (path length = 1) between two vertices $v_1$ and $v_2$, higher powers of the adjacency matrix give us insight into the number of paths between $v_1$ and $v_2$ of length $k$. A graph with $N$ vertices where all $\binom{N}{2}$ (pairwise) possible edges are included is called a \textbf{complete} graph, and is denoted by $K_N$. In the case $N=3$, we call $K_3$ a triangle.

\begin{corollary} Let $A$ be an $N\times N$ adjacency matrix. The $ii^{th}$ diagonal entry of $A^3$ is the number of triangles in $A$ containing the vertex $i$. The number of triangles in $A$ is $\frac{\text{Trace}(A^3)}{6}$, the sum of the diagonal entries of $A^3$, taking into account over-counting.
\end{corollary}

\noindent \textbf{Example of Corollary 2:} Suppose we have a dataset with size $N=6$. The number of triangles that exist in the complete $K_6$ graph is
\begin{align*}
\frac{\text{Trace}(A(G_6)^3)}{6} &=abf +acg +adh +aei +bcj \\
                          		 &+fgj +bdk +fhk +bel +fil \\
                          		 &+cdm +ghm +jkm +cen +gin \\
                          		 &+jln +deo +hio +klo +mno
\end{align*}
, where each triplet $(e_i,e_j,e_k)$ is a triplet of edges that create a triangle ($K_3$). Depending on whether or not each edge has a value of $1$ or $0$ in the adjacency matrix $A$ will determine if these triangles exist. 

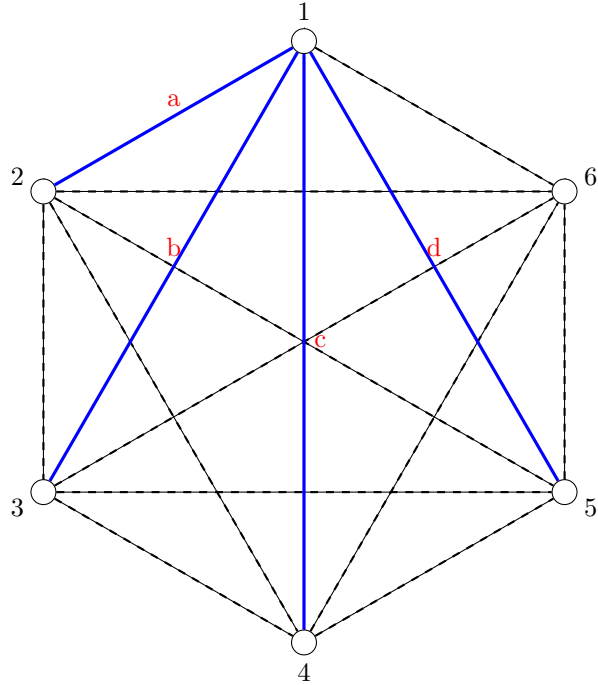
\begin{figure}[H]
\caption{A graph where only the edges $a,b,c$ and $d$ exist. The dashed lines are used to indicate a lack of edge.}
\begin{center}
\begin{tikzpicture}
  \graph[circular placement, radius=4cm,
         empty nodes, nodes={circle,draw}] {
    \foreach \x in {1,...,6} {
      \foreach \y in {\x,...,6} {
        \x -- \y;
      };
    };
  };
  \foreach \x [count=\idx from 0] in {1,...,6} {
    \pgfmathparse{90 + \idx * (360 / 6)}
    \node at (\pgfmathresult:4.4cm) {\x};
  };
  	\begin{scope}   [dashed]
      \draw (1) edge[black, thick]  (6);
      \draw (2) edge[black, thick] (3);
      \draw (2) edge[black, thick] (4);
      \draw (2) edge[black, thick] (5);
      \draw (2) edge[black, thick] (6);
      \draw (3) edge[black, thick] (4);
      \draw (3) edge[black, thick] (5);
      \draw (3) edge[black, thick] (6);
      \draw (4) edge[black, thick] (5);
      \draw (4) edge[black, thick] (6);
      \draw (5) edge[black, thick] (6);
    \end{scope}
	\draw (1) edge[blue, very thick] node[red, above] {a} (2);
	\draw (1) edge[blue, very thick] node[red,above] {b} (3);
	\draw (1) edge[blue, very thick] node[red,right] {c} (4);
	\draw (1) edge[blue, very thick] node[red,above] {d} (5);
    
\end{tikzpicture}
\end{center}
\end{figure}

Suppose only the edges $a,b,c,$ and $d$ exist. Then no triangles exist because when we replace the edges $a,b,c,$ and $d$ with 1 and everything else with $0$, no triplet of edges is complete:
\begin{align*}
\frac{\text{Trace}(A(G_6)^3)}{6} &= (1)(1)(0) +(1)(1)(0)+(1)(1)(0)+(1)(0)(0)+(1)(1)(0) \\
                          		 &+ (0)(0)(0) +(1)(1)(0)+(0)(0)(0)+(1)(1)(0)+(0)(0)(0) \\
                          		 &+ (1)(1)(0) +(0)(0)(0)+(0)(0)(0)+(1)(0)(0)+(0)(0)(0) \\
                          		 &+ (0)(0)(0) +(1)(0)(0)+(0)(0)(0)+(0)(0)(0)+(0)(0)(0) \\
								 &= 0
\end{align*}

In this framework, if a triangle exists in the adjacency matrix $A$, then all three points $(v_i,v_j,v_k)$ are connected to each other based on how the predetermined relationship is defined (whether it be geographic distance or some measurement of friendship, for example). In this way, a $K_3$ represents the simplest non-trivial emergent ``pattern" that can be observed in a graph connecting data points in $D$, so it's the natural starting point for asking the question, ``Which patterns are forced to exist in $D$ given how we've connected its data points in the adjacency matrix $A$?".

This framework is good in black-and-white, binary situations where any pair of vertices is either (strongly) related or not related (at all). In non-binary relationships, it can be useful to think about graphs whose edges are classified by multiple colors. This can be represented as a partition of the edge set $E$ into $r$-many disjoint sets $E = E_{c_1} \sqcup E_{c_2} \sqcup \ldots \sqcup E_{c_r}$, where $c_1,...,c_r$ represent a total of $r-$colors or classifications.

In the case of two colors, we will often just refer to red (R) and blue (B) edges. In the framework of adjacency matrices, a complete graph $A$ with an edge-coloring using two colors is represented by an adjacency matrix $B$ indicating a relationship exists or does not $R$:

\begin{equation}
{R(G_6) = \bordermatrix{~
	 &v_1 & v_2 & v_3&v_4&v_5&v_6\cr
 v_1&0 & a & b&c&d&e\cr
 v_2&a & 0 & f&g&h&i\cr
 v_3&b & f & 0&j&k&l\cr
 v_4&c & g & j&0&m&n\cr
 v_5&d & h & k&m&0&o \cr
 v_6&e & i & l&n&o&0
 }}\end{equation} 
\begin{equation}{B(G_6) = \bordermatrix{~
	 &v_1 & v_2 & v_3&v_4&v_5&v_6\cr
 v_1&0 & 1-a & 1-b&1-c&d&1-e\cr
 v_2&1-a & 0 & 1-f&1-g&1-h&1-i\cr
 v_3&1-b & 1-f & 0&1-j&1-k&1-l\cr
 v_4&1-c & 1-g & 1-j&0&1-m&1-n\cr
 v_5&1-d & 1-h & 1-k&1-m&0&1-o \cr
 v_6&1-e & 1-i & 1-l&1-n&1-o&0
 }}.
\end{equation}
$$\therefore R(G_6)+B(G_6) = {\bordermatrix{~
	 &v_1 & v_2 & v_3&v_4&v_5&v_6\cr
 v_1&0 & 1 & 1&1&1&1\cr
 v_2&1 & 0 & 1&1&1&1\cr
 v_3&1 & 1 & 0&1&1&1\cr
 v_4&1 & 1 & 1&0&1&1\cr
 v_5&1 & 1 & 1&1&0&1 \cr
 v_6&1 & 1 & 1&1&1&0
 }} $$

Take the edge $a$ between $v_1$ and $v_2$ in $R$ and set it equal to $a=1$. Since the edge is colored red, it necessarily has to have an entry equal to zero $(a-1 = 1-1 = 0) $ in the blue edge adjacency graph $B$. In this case $R + B$ must be the matrix of all ones, except on the diagonal where it has zeros. Counting monochromatic triangles in $A$ is particularly simple:
\begin{corollary}\label{cor:counting_triangles} Let $A$ be an $N\times N$ adjacency matrix whose edges are colored using two colors. The number of monochromatic triangles in $A$ is $\frac{\text{Trace}(B^3) + \text{Trace}(R^3)}{6}$.
\end{corollary}

Therefore, the total number of triangles in the dataset $D$ is equal to the sum of red and blue triangles present in the adjacency matrices $R$ and $B$.

\section{The Ramsey perspective}
\label{sec:Ramsey}

Classical Ramsey theory asks: ``Fix $m, r$. Does every edge coloring of a $K_N$ complete graph with $r$ colors contain a sub-collection $K_m$, all of which have the same color?'' In other words, how big does a multi-colored, complete graph need to be to force the existence of a smaller single-colored, complete graph?

In 1929, Ramsey  \cite{Ramsey1930} showed that if the size of the dataset $D$ was $N\geq6$, and the number of ways the data points could be related to each other was $m=2$ (either related or unrelated), then unavoidable subgraphs of mutually related or unrelated data points are forced to exist.

In 1959, Goodman quantified how many single-colored (monochromatic) triangles must be present in a two-colored $K_N$. Because a ($K_3$) represents the simplest object that describes how data points relate to each other beyond a simple edge, it will form the basis of our application of Ramsey theory.

\begin{theorem}[Goodman 1959, \cite{goodman1959}]\label{thm:goodman} Let $G$ be a graph with $N$ vertices and edge-colored with red and blue. The quantity of monochromatic triangles in $G$ is at least:
\begin{itemize}

	\item $\frac{m(m-1)(m-2)}{3}$, if $N=2m$, 
	\item $\frac{2m(m-1)(4m+1)}{3}$, if $N=4m+1$,
	\item $\frac{2m(m+1)(4m-1)}{3}$, if $N=4m+3$.
\end{itemize}
\end{theorem}
Since the total number of triangles in $K_N$ is $\binom{N}{3} = \frac{N(N-1)(N-2)}{6}$, Goodman's formula may be reinterpreted as a percentage. 
\begin{corollary}[Goodman 1959, \cite{goodman1959}]\label{thm:goodman_percentage} Let $G$ be a graph with $N$ vertices and edge-colored with red and blue. The percentage of triangles in $G$ that are monochromatic is asymptotically at least $\frac{N-3}{4 N} \rightarrow \frac{1}{4}$.
\end{corollary}
This can be shown directly by dividing the quantities in Theorem \ref{thm:goodman} by $\binom{N}{3}$. Alternatively, by applying Schwenk's reformulation of Goodman's formula \cite{schwenk1972}, we can easily prove this: 
$$ $$
\begin{proof} 
For $N$ number of data points, the forced number $N$ of monochromatic red (R) and blue (B) triangles is:
\[
F(N)= \binom{N}{3} - \left\lfloor{\frac{1}{2}N \left\lfloor{ \frac{1}{4}(N-1)^2}\right\rfloor}\right\rfloor,
\]
and since the number of triangles present in any complete graph is $\binom{N}{3}$, the following ratio describes the percentage of triangles in $G$ that are monochromatic:
$$1- \frac{\left\lfloor \frac{1}{2} N \left\lfloor \frac{1}{4} (N-1)^2\right\rfloor \right\rfloor }{\binom{N}{3}}=1 - \frac{6 \left\lfloor \frac{1}{2} N
\left\lfloor \frac{1}{4} (N-1)^2\right\rfloor \right\rfloor }{(N-2) (N-1) N}$$
The Floor Function of $f(x)$ is equivalent to the function of $f(x)$ with discontinuities at non-integer values $x$, therefore describing the asymptotic nature of the above ratio can be done without taking the floor functions into consideration:
$$1 - \frac{6 \frac{1}{2} N (\frac{1}{4} (N-1)^2)}{(N-2) (N-1) N} = \frac{1}{4} - \frac{3}{4(N-2)} $$
$$\therefore \lim_{n\rightarrow \infty} \frac{1}{4} - \frac{3}{4(N-2)} = \frac{1}{4} $$
\end{proof}

From this we can establish a threshold for when a two-colored graph can be interpreted to have meaningful correlations.

\begin{definition} Let $G$ be a graph with $n$ vertices and edge-colored with red and blue. 
Let $\text{Mono}(G)$ be the percentage of triangles in $G$ that are monochromatic, among all possible $\binom{N}{3}$ triangles in $G$.
Let $\text{Goodman}(N)$ be the minimum percentage of monochromatic triangles in $G$ guaranteed by Corollary \ref{thm:goodman_percentage}, which has been shown to approach 0.25 as $N\rightarrow \infty$.
If $\text{Mono}(G) > \text{Goodman}(N)$ then we say that $G$ has \textbf{potentially meaningful correlations}, which we explore further in section 4.3.
\end{definition}

If $\text{Mono}(G)$ is much larger than $\text{Goodman}(N)$, then we might say that $G$ obeys a \textbf{triangle dichotomy}, which means that we expect a lot of triangles to be either completely one color, or completely the other. This is a relative term, and the larger $\text{Mono}(G)$ is the more that this resembles a true dichotomy. If one color is more heavily represented in G than another, then we might say that $G$ has a \textbf{triangle bias}. When triangle bias exists, this is at odds with the expectation, in a randomly colored graph, that the ratio of the number of color R triangles to color B triangles should be $1:1$, and is therefore a further indication that the correlations in $G$ are meaningful.

How can this be used in a dataset? In section 4.3, we discuss a best-fit approach in order to test the null-hypothesis that a dataset is indeed random. 

\section{Models}
\label{sec:Models}

We will focus our analysis on two datasets: (1) voting records of members of the US congress in 1984, and (2) economic partnership among countries.

\subsection{Similarity in voting records}
\label{subec:voting}

We now look at a set of people that have voted multiple times, specifically the 1984 United States Congressional Voting Records \cite{Lichman:2013}.

Goodman's formula can quantify how strong the triangle dichotomy and triangle bias are; that is, the percentage of three person cliques (B) and independent triples (R) and their deviation from the expected $1:1$ color ratio. We will use the Hamming distance to measure how similar two voting records are for the 435 congress members of the 1984 US congress.

\begin{definition} The Hamming distance of two strings of the same length is the total number of positions where the entries are different.
\end{definition}

For example, the Hamming distance between $0\textbf{0}0\textbf{10}$ and $0\textbf{1}0\textbf{01}$ is $3$. These strings differ in the second, fourth and fifth spots.

 In this session there were 16 separate votes, and to each voter we assign the string of length 16 with entries `N' (voted nay), `Y' (voted yea) or `A' (some other action, such as abstaining). The minimum Hamming distance is 0, which indicates two identical voting records, and the maximum distance is 16, meaning the two voters always voted differently. 

\begin{table}[h!]
\centering
\label{table:congress_voters_sample}
\caption{As a small example, the first 6 voters have the following strings associated with them.}
\begin{tabular}{llllll}
Voter & Party &  &  Voting string&  &  \\
$v_1 $& R & NYNY & YYNN & NYAY & YYNY \\
$v_2 $& R & NYNY & YYNN & NNNY & YYNA \\
$v_3 $& D & AYYA & YYNN & NNYN & YYNN \\
$v_4$ & D & NYYN & AYNN & NNYA & YNNY \\
$v_5 $& D & YYYN & YYNN & NNYA & YYYY \\
$v_6$ & D & NYYN & YYNN & NNNN & YYYY
\end{tabular}
\end{table}

Applying this notion of distance to the data from Table \ref{table:congress_voters_sample} gives the following adjacency matrix and graph:

\begin{equation}
\bordermatrix{~
	&v_1 & v_2 & v_3&v_4&v_5&v_6\cr
v_1 & 0 & 3 & 7 & 7 & 7 & 6  \cr
v_2 & 3 & 0 & 6 & 7 & 7 & 5  \cr
v_3 & 7 & 6 & 0 & 5 & 5 & 5  \cr
v_4 & 7 & 7 & 5 & 0 & 5 & 4  \cr
v_5 & 7 & 7 & 5 & 5 & 0 & 3  \cr
v_6 & 6 & 5 & 5 & 4 & 3 & 0
 }\end{equation}
\begin{figure}[H]
\centering
\scalebox{0.75}{
\begin{tikzpicture}
  \graph[circular placement, radius=4cm,
         empty nodes, nodes={circle,draw}] {
    \foreach \x in {1,...,6} {
      \foreach \y in {\x,...,6} {
        \x -- \y;
      };
    };
    1 --["3"'] 2;
    1 --["7"' near start] 3;
	1 --["7"' near start] 4;
    1 --["7", near start] 5;
    1 --["6",] 6;
    2 --["6"'] 3;
	2 --["7"' near start] 4;
    2 --["7", near start] 5;
    2 --["5", near start] 6;
	3 --["5"'] 4;
    3 --["5", near start] 5;
    3 --["5", near start] 6;
	4 --["5",] 5;
    4 --["4", near start] 6;
	5 --["3",] 6;
  };
  \foreach \x [count=\idx from 0] in {1,...,6} {
    \pgfmathparse{90 + \idx * (360 / 6)}
    \node at (\pgfmathresult:4.4cm) {\x};
  };
\end{tikzpicture}
}
\end{figure}

How do we turn this into a binary adjacency matrix for classification purposes? In other words, how do we decide what constitutes ``similar" voting and ``dissimilar" voting, and where do we make that cut off?

\begin{definition}\label{def:threshold} Let $(M,d)$ be a metric space with vertex set $M$ and distance $d$. Let $t \geq 0$. Define the two-colored threshold graph $G(t)$ in the space $(M,d)$ by coloring an edge between two vertices $v_i,v_j$ blue $B$ iff $d(v_i,v_j) > t$, and red $R$ iff $d(v_i,v_j) \leq t$.
\end{definition}

 Therefore, when for example $t=5$, the following graph $G(t=5)$ maps to the binary case: 
\begin{equation}
\bordermatrix{~
	&v_1 & v_2 & v_3&v_4&v_5&v_6\cr
v_1 & 0 & 3 & 7 & 7 & 7 & 6  \cr
v_2 & 3 & 0 & 6 & 7 & 7 & 5  \cr
v_3 & 7 & 6 & 0 & 5 & 5 & 5  \cr
v_4 & 7 & 7 & 5 & 0 & 5 & 4  \cr
v_5 & 7 & 7 & 5 & 5 & 0 & 3  \cr
v_6 & 6 & 5 & 5 & 4 & 3 & 0
 }\rightarrow B(G_6) =\bordermatrix{~
	&v_1 & v_2 & v_3&v_4&v_5&v_6\cr
v_1 & 0 & 0 & 1 & 1 & 1 & 1  \cr
v_2 & 0 & 0 & 1 & 1 & 1 & 0  \cr
v_3 & 1 & 1 & 0 & 0 & 0 & 0  \cr
v_4 & 1 & 1& 0 & 0 & 0 & 0  \cr
v_5 & 1 & 1 & 0 & 0 & 0 & 0  \cr
v_6 & 1 & 0 & 0 & 0 & 0 & 0
 } \end{equation}
\begin{equation}
\bordermatrix{~
	&v_1 & v_2 & v_3&v_4&v_5&v_6\cr
v_1 & 0 & 3 & 7 & 7 & 7 & 6  \cr
v_2 & 3 & 0 & 6 & 7 & 7 & 5  \cr
v_3 & 7 & 6 & 0 & 5 & 5 & 5  \cr
v_4 & 7 & 7 & 5 & 0 & 5 & 4  \cr
v_5 & 7 & 7 & 5 & 5 & 0 & 3  \cr
v_6 & 6 & 5 & 5 & 4 & 3 & 0
 }\rightarrow R(G_6) =\bordermatrix{~
	&v_1 & v_2 & v_3&v_4&v_5&v_6\cr
v_1 & 0 & 1 & 0 & 0 & 0 & 0  \cr
v_2 & 1 & 0 & 0 & 0 & 0 & 1  \cr
v_3 & 0 & 0 & 0 & 1 & 1 & 1  \cr
v_4 & 0 & 0 & 1 & 0 & 1 & 1  \cr
v_5 & 0 & 0 & 1 & 1 & 0 & 1  \cr
v_6 & 0 & 1 & 1 & 1 & 1 & 0
 } \end{equation}

For example, taking the sample Congressional voting data from Table \ref{table:congress_voters_sample}, we get the following threshold graphs in Figures \ref{fig:threshold1}, \ref{fig:threshold2} and \ref{fig:threshold3}.

\begin{figure}[H]
\centering
\scalebox{0.5}{
\begin{tikzpicture}
  \graph[circular placement, radius=4cm,
         empty nodes, nodes={circle,draw}] {
    \foreach \x in {1,...,6} {
      \foreach \y in {\x,...,6} {
        \x -- \y;
      };
    };
  };
  \foreach \x [count=\idx from 0] in {1,...,6} {
    \pgfmathparse{90 + \idx * (360 / 6)}
    \node at (\pgfmathresult:4.4cm) {\x};
  };
	\draw (1) edge[blue, thick] (2);
	\draw (1) edge[blue, thick] (3);
	\draw (1) edge[blue, thick] (4);
	\draw (1) edge[blue, thick] (5);
	\draw (1) edge[blue, thick] (6);
	\draw (2) edge[blue, thick] (3);
	\draw (2) edge[blue, thick] (4);
	\draw (2) edge[blue, thick] (5);
	\draw (2) edge[blue, thick] (6);
	\draw (3) edge[blue, thick] (4);
	\draw (3) edge[blue, thick] (5);
	\draw (3) edge[blue, thick] (6);
	\draw (4) edge[blue, thick] (5);
	\draw (4) edge[blue, thick] (6);
	\draw (5) edge[blue, thick] (6);
\end{tikzpicture}
\begin{tikzpicture}
  \graph[circular placement, radius=4cm,
         empty nodes, nodes={circle,draw}] {
    \foreach \x in {1,...,6} {
      \foreach \y in {\x,...,6} {
        \x -- \y;
      };
    };
  };
  \foreach \x [count=\idx from 0] in {1,...,6} {
    \pgfmathparse{90 + \idx * (360 / 6)}
    \node at (\pgfmathresult:4.4cm) {\x};
  };
	\draw (1) edge[red, thick] (2);
	\draw (1) edge[blue, thick] (3);
	\draw (1) edge[blue, thick] (4);
	\draw (1) edge[blue, thick] (5);
	\draw (1) edge[blue, thick] (6);
	\draw (2) edge[blue, thick] (3);
	\draw (2) edge[blue, thick] (4);
	\draw (2) edge[blue, thick] (5);
	\draw (2) edge[blue, thick] (6);
	\draw (3) edge[blue, thick] (4);
	\draw (3) edge[blue, thick] (5);
	\draw (3) edge[blue, thick] (6);
	\draw (4) edge[blue, thick] (5);
	\draw (4) edge[blue, thick] (6);
	\draw (5) edge[red, thick] (6);
\end{tikzpicture}
}
\caption{$G(t=0), G(1), G(2)$ (left) and $G(3)$ (right)}
\label{fig:threshold1}
\end{figure}
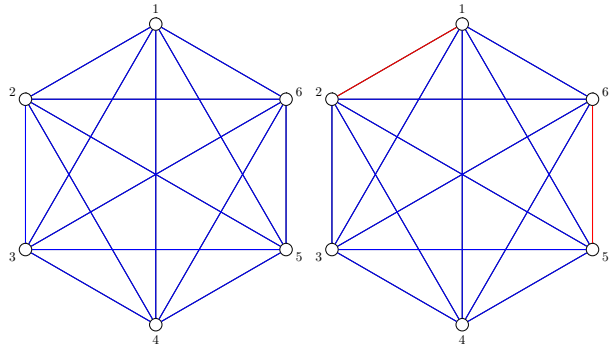

\begin{figure}[H]
\centering
\scalebox{0.5}{
\begin{tikzpicture}
  \graph[circular placement, radius=4cm,
         empty nodes, nodes={circle,draw}] {
    \foreach \x in {1,...,6} {
      \foreach \y in {\x,...,6} {
        \x -- \y;
      };
    };
  };
  \foreach \x [count=\idx from 0] in {1,...,6} {
    \pgfmathparse{90 + \idx * (360 / 6)}
    \node at (\pgfmathresult:4.4cm) {\x};
  };
	\draw (1) edge[red, thick] (2);
	\draw (1) edge[blue, thick] (3);
	\draw (1) edge[blue, thick] (4);
	\draw (1) edge[blue, thick] (5);
	\draw (1) edge[blue, thick] (6);
	\draw (2) edge[blue, thick] (3);
	\draw (2) edge[blue, thick] (4);
	\draw (2) edge[blue, thick] (5);
	\draw (2) edge[blue, thick] (6);
	\draw (3) edge[blue, thick] (4);
	\draw (3) edge[red, thick] (5);
	\draw (3) edge[blue, thick] (6);
	\draw (4) edge[blue, thick] (5);
	\draw (4) edge[red, thick] (6);
	\draw (5) edge[red, thick] (6);
\end{tikzpicture}
\begin{tikzpicture}
  \graph[circular placement, radius=4cm,
         empty nodes, nodes={circle,draw}] {
    \foreach \x in {1,...,6} {
      \foreach \y in {\x,...,6} {
        \x -- \y;
      };
    };
  };
  \foreach \x [count=\idx from 0] in {1,...,6} {
    \pgfmathparse{90 + \idx * (360 / 6)}
    \node at (\pgfmathresult:4.4cm) {\x};
  };
	\draw (1) edge[red, thick] (2);
	\draw (1) edge[blue, thick] (3);
	\draw (1) edge[blue, thick] (4);
	\draw (1) edge[blue, thick] (5);
	\draw (1) edge[blue, thick] (6);
	\draw (2) edge[blue, thick] (3);
	\draw (2) edge[blue, thick] (4);
	\draw (2) edge[blue, thick] (5);
	\draw (2) edge[red, thick] (6);
	\draw (3) edge[red, thick] (4);
	\draw (3) edge[red, thick] (5);
	\draw (3) edge[red, thick] (6);
	\draw (4) edge[red, thick] (5);
	\draw (4) edge[red, thick] (6);
	\draw (5) edge[red, thick] (6);
\end{tikzpicture}
}
\caption{$G(4)$ and $G(5)$}
\label{fig:threshold2}
\end{figure}
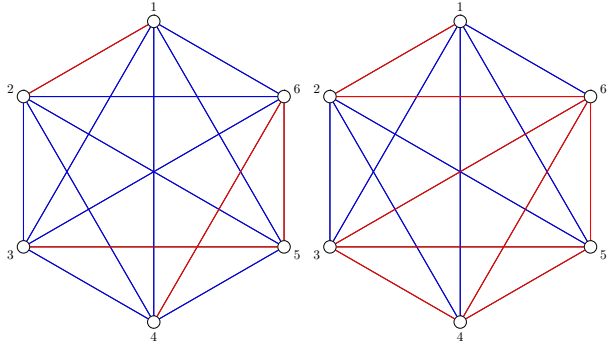

\begin{figure}[H]
\centering
\scalebox{0.5}{
\begin{tikzpicture}
  \graph[circular placement, radius=4cm,
         empty nodes, nodes={circle,draw}] {
    \foreach \x in {1,...,6} {
      \foreach \y in {\x,...,6} {
        \x -- \y;
      };
    };
  };
  \foreach \x [count=\idx from 0] in {1,...,6} {
    \pgfmathparse{90 + \idx * (360 / 6)}
    \node at (\pgfmathresult:4.4cm) {\x};
  };
	\draw (1) edge[red, thick] (2);
	\draw (1) edge[blue, thick] (3);
	\draw (1) edge[blue, thick] (4);
	\draw (1) edge[blue, thick] (5);
	\draw (1) edge[red, thick] (6);
	\draw (2) edge[red, thick] (3);
	\draw (2) edge[blue, thick] (4);
	\draw (2) edge[blue, thick] (5);
	\draw (2) edge[red, thick] (6);
	\draw (3) edge[red, thick] (4);
	\draw (3) edge[red, thick] (5);
	\draw (3) edge[red, thick] (6);
	\draw (4) edge[red, thick] (5);
	\draw (4) edge[red, thick] (6);
	\draw (5) edge[red, thick] (6);
\end{tikzpicture}
\begin{tikzpicture}
  \graph[circular placement, radius=4cm,
         empty nodes, nodes={circle,draw}] {
    \foreach \x in {1,...,6} {
      \foreach \y in {\x,...,6} {
        \x -- \y;
      };
    };
  };
  \foreach \x [count=\idx from 0] in {1,...,6} {
    \pgfmathparse{90 + \idx * (360 / 6)}
    \node at (\pgfmathresult:4.4cm) {\x};
  };
	\draw (1) edge[red, thick] (2);
	\draw (1) edge[red, thick] (3);
	\draw (1) edge[red, thick] (4);
	\draw (1) edge[red, thick] (5);
	\draw (1) edge[red, thick] (6);
	\draw (2) edge[red, thick] (3);
	\draw (2) edge[red, thick] (4);
	\draw (2) edge[red, thick] (5);
	\draw (2) edge[red, thick] (6);
	\draw (3) edge[red, thick] (4);
	\draw (3) edge[red, thick] (5);
	\draw (3) edge[red, thick] (6);
	\draw (4) edge[red, thick] (5);
	\draw (4) edge[red, thick] (6);
	\draw (5) edge[red, thick] (6);
\end{tikzpicture}
}
\caption{$G(6)$ and $G(7), \ldots, G(17)$ }
\label{fig:threshold3}
\end{figure}
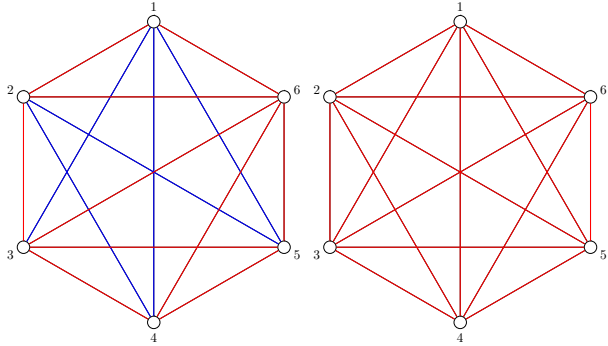

Consider the complete graph $G(t)$, the subgraph composed entirely of Democrat Congress votes $D(t)$, and the subgraph composed entirely of Republican Congress votes $R(t)$. Applied to the total voting records available in the dataset at various thresholds $t$, the following table shows the ratio of $\text{Mono}(G(t))$ to the total number of triangles $K_N$:

\begin{table}[H]
\centering
\caption{The percentage of monochromatic triangles for various threshold graphs. The minimum values are boxed.}
\begin{tabular}{llll}
 t & $G(t)$       & $D(t)$ 		 & $R(t)$ 		\\ \hline
 0 & 1.000        & 1.000        & 1.000        \\
 1 & 0.993        & 0.933        & 0.972        \\
 2 & 0.953        & 0.953        & 0.829        \\
 3 & 0.858        & 0.850        & 0.603        \\
 4 & 0.727        & 0.699        & 0.475        \\
 5 & 0.590        & 0.549        & \fbox{0.461} \\
 6 & 0.462        & 0.440        & 0.506        \\
 7 & 0.359        & \fbox{0.399} & 0.581        \\
 8 & 0.291        & 0.423        & 0.672        \\
 9 & \fbox{0.271} & 0.496        & 0.770        \\
10 & 0.299        & 0.586        & 0.842        \\
11 & 0.370        & 0.688        & 0.891        \\
12 & 0.471        & 0.783        & 0.932        \\
13 & 0.597        & 0.871        & 0.954        \\
14 & 0.743        & 0.943        & 0.964        \\
15 & 0.888        & 0.979        & 0.977        \\
16 & 0.970        & 0.997        & 0.984        \\
17 & 1.000        & 1.000        & 1.000        \\\hline
Goodman & 0.248   & 0.247 		 & 0.246 		\\
\end{tabular}
\end{table}

\begin{tikzpicture}[scale=.75]

\begin{axis}[legend style={at={(1,1)},anchor=south east}]
\addplot[black,mark=otimes] table [x=a, y=b, col sep=comma] {data.csv};
\addplot[blue,mark=diamond] table [x=a, y=c, col sep=comma] {data.csv};
\addplot[red,mark=triangle] table [x=a, y=d, col sep=comma] {data.csv};
\addplot[black,sharp plot,update limits=false] coordinates {(-3,0.2483) (20,0.2483)};
\addplot[blue,sharp plot,update limits=false] coordinates {(-3,0.2472) (20,0.2472)};
\addplot[red,sharp plot,update limits=false] coordinates {(-3,0.2456) (20,0.2456)};
\legend{All Monochromatic Triangles in G( t ),Democrats D( t ),Republicans R( t ),Goodman}

\end{axis}
\end{tikzpicture}
\hspace{0cm}
\begin{tikzpicture}[scale=.75]
\begin{axis}[legend style={at={(1,1)},anchor=south east}]
\addplot[black,mark=otimes] table [x=a, y=b, col sep=comma] {data.csv};
\addplot[red,mark=diamond] table [x=a, y=h, col sep=comma] {data.csv};
\addplot[blue,mark=triangle] table [x=a, y=i, col sep=comma] {data.csv};
\addplot[black,sharp plot,update limits=false] coordinates {(-3,0.2483) (20,0.2483)};
\legend{All Monochromatic Triangles in G( t ), Red Triangles in G( t ), Blue Triangles in G( t ),Goodman}
\end{axis}
\end{tikzpicture}

\begin{tikzpicture}[scale=.75]
\centering
\begin{axis}[legend style={at={(1,1)},anchor=south east}]
\addplot[black,mark=otimes] table [x=a, y=c, col sep=comma] {data.csv};
\addplot[red,mark=diamond] table [x=a, y=j, col sep=comma] {data.csv};
\addplot[blue,mark=triangle] table [x=a, y=k, col sep=comma] {data.csv};
\addplot[blue,sharp plot,update limits=false] coordinates {(-3,0.2472) (20,0.2472)};
\legend{All Monochromatic Triangles in D( t ), Red Triangles in D( t ), Blue Triangles in D( t ),Goodman}
\end{axis}
\end{tikzpicture}
\hspace{0cm}
\begin{tikzpicture}[scale=.75]
\begin{axis}[legend style={at={(1,1)},anchor=south east}]
\addplot[black,mark=otimes] table [x=a, y=d, col sep=comma] {data.csv};
\addplot[red,mark=diamond] table [x=a, y=l, col sep=comma] {data.csv};
\addplot[blue,mark=triangle] table [x=a, y=m, col sep=comma] {data.csv};
\addplot[red,sharp plot,update limits=false] coordinates {(-3,0.2456) (20,0.2456)};
\legend{All Monochromatic Triangles in R( t ), Red Triangles in R( t ), Blue Triangles in R( t ),Goodman}
\end{axis}
\end{tikzpicture}

In section \ref{sec:trans} we give another natural interpretation of these results by giving a measure of how transitive these graphs are. This is maybe a more intuitive interpretation of the data since it gives us a direct measurement of cooperation and independence.

\subsection{How to measure the deviation away from a random graph}

Goodman's formula tells us how many monochromatic triangles are forced to exist for a dataset $D$ of size $N$, but what would the threshold graph of a truly random coloring of $G_N$ look like? 

\subsubsection{Theoretical Construction}Suppose we have a random graph $G(N,t)$ of  related to $N$ data points in D and probability $t$ that an edge between two vertices $n_i,n_j$ exists, the Erdos-Renyi model tells us that the expected number of edges in $G(N,t)$ is $\binom{N}{2}t$. The parameter $t$ can be thought of as the threshold parameter introduced in section \ref{subec:voting} as it ranges from $0\rightarrow1$ (assuming $t_{min}$ and $t_{max}$ have been normalized to  $[0,1]$). Therefore, the expected number of red ($R$) and blue ($B$) triangles $T$ in $D$ is:

$$\mathbb{E}[T^R]=\binom{N}{3}(t^3), \mathbb{E}[T^B]=\binom{N}{3}(1-t)^3 $$
, where $\binom{N}{3}$ represents the total possible number of triangles, and $t^3$ represents the probability that all three edges are red (R); likewise $(1-t)^3$ represents the probability of all three edges being blue (B). This information can be used to calculate the number of monochromatic triangles.

\begin{corollary} The expected number of monochromatic triangles in $G_N$ is $2\binom{N}{3}(\frac{1}{2})^3$.
\end{corollary}

\textbf{Proof.} The probability that 3 adjacent edges are the same color in a 2-colored graph is $(\frac{1}{2})^3$, there are $\binom{N}{3}$ number of triangles, and we multiply by 2 to account for the symmetry of how the edges can be colored with equal probability. 
This creates the following threshold plot for any randomly colored graph $G(N,t)$:

\begin{figure}[H]
\centering
\label{fig:expected_value}
\includegraphics[scale=0.35]{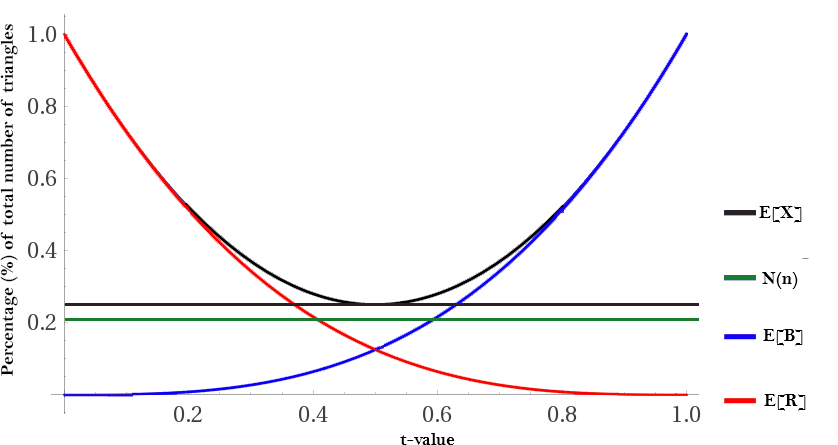}
\caption{An example using $N=20$}
\end{figure}
\subsubsection{Defining Deviation}
\begin{definition}
\textbf{Deviation} is a combination of the degree of (A): the triangle dichotomy and (B): the triangle bias.
\end{definition}
Deviation away from the expected distribution can allow us to determine the likelihood that the null hypothesis $H_0$ (that $G_N$ is actually random) is accepted or rejected. This can be done with a simple $\chi^2$ test: 

$$\chi^2_R = \sum_{i=t_{min}}^{t_{max}} \frac{(R(i)_O - E[R]_i)^2 }{E[R]_i},\chi^2_B = \sum_{i=t_{min}}^{t_{max}} \frac{(B(i)_O - E[ B]_i)^2 }{E[B]_i} $$

The average of $\chi^2_R,\chi^2_B$ and their resulting $p-$value can be used to determine with some significance level whether to accept or reject $H_0$.

While the expected value is a good benchmark, it still doesn't answer the more fundamental question of how many monochromatic triangles are present in $G_N$ versus how many are required by Ramsey theory. This creates a stricter $\chi^2$ calculation, but one that's better suited to our needs and is a measurement of the \textbf{triangle dichotomy} and \textbf{triangle bias}:
$$\chi^2_G = \sum_{i=t_{min}}^{t_{max}} \frac{(G(i)_O - F(N))^2 }{F(N)}$$

$$\chi^2_R = \sum_{i=t_{min}}^{t_{max}} \frac{(R(i)_O - \frac{F(N)}{2})^2 }{\frac{F(N)}{2}},\chi^2_B = \sum_{i=t_{min}}^{t_{max}} \frac{(B(i)_O - \frac{F(N)}{2})^2 }{\frac{F(N)}{2}}$$

\subsubsection{Applied to voting threshold graphs}

We are now faced with applying the $\chi^2$ method from 4.2.2 to the Congressional voting threshold graphs. What is the likelihood that these are random, or equivalently, what is the likelihood that there is a bias in the congressional voting record?

\begin{table}[H]
\centering
\caption{\small{The $\chi^2$ fit for the overall voting record G(t), Democrats D(t), and Republicans R(t). This demonstrates the degree of the \textbf{triangle dichotomy} for each pre-defined classification.}}
\begin{tabular}{llll}
 & $G(t)$       & $D(t)$ 		 & $R(t)$ 		\\ \hline

Goodman & 0.248   & 0.247 		 & 0.246 		\\
$\chi^2$ & 17.448   & 22.552		 & 25.206		\\
\end{tabular}
\end{table}

\begin{table}[H]
\centering
\caption{\small{The $\chi^2$ fit for the overall voting record G(t), Democrats D(t), and Republicans R(t) by color (R,B). This demonstrates the degree of the \textbf{triangle bias} for each pre-defined classification.}}
\begin{tabular}{llll}
 & Blue $\chi^2$        & Red $\chi^2$ 		 & Total $\chi^2$ 		\\ \hline

Total G(t) & 18.7782  & 22.864		 & 17.448		\\
Democrats D(t) & 28.536   & 22.596	 & 22.552		\\
Republicans R(t) &22.705  & 38.028	 & 25.206	\\
\end{tabular}
\end{table}

These $\chi^2$ values have $p-$values that are very, very small. A way to place these in context is to compare them to the expected value's deviation from what's required by Ramsey theory:

\begin{table}[H]
\centering
\caption{\small{The $\chi^2$ fit for the overall expected value of forced monochromatic triangles.}}
\begin{tabular}{llll}
 & Blue $\chi^2$        & Red $\chi^2$ 		 & Total $\chi^2$ 		\\ \hline

Expectation & 16.384 & 16.384	 & 10.076	\\
\end{tabular}
\end{table}

\begin{table}[H]
\centering
\caption{The deviation of $\chi^2$ of G(t), D(t), and R(t) from their respective expected $\chi^2$ values. }
\begin{tabular}{llll}
 & Blue $\chi^2$        & Red $\chi^2$ 		 & Total $\chi^2$ 		\\ \hline

Total &  $|18.7782-16.384|=2.394$  & $6.48$		 & $|17.448-10.076|= 7.372$		\\
Democrats &12.152   & 6.212	 & 12.476		\\
Republicans &6.321  & 21.644	 & 15.130	\\
\end{tabular}
\end{table}

The p-value associate with each of these is based on the cumulative distribution function (CDF), namely $p = 1-CDF$:

\begin{table}[H]
\centering
\caption{}
\begin{tabular}{llll}
 & Blue $\chi^2$        & Red $\chi^2$ 		 & Total $\chi^2$ 		\\ \hline

Total &  $ \underline{0.121802}$  & $\underline{0.010909}$		 & $0.006625$		\\
Democrats &0.00049   &  \underline{0.012689}	 & 0.000412		\\
Republicans &\underline{0.011932}  & $<$ 0.00001	 &  \textbf{0.0001}	\\
\end{tabular}
\end{table}

We can say a p-value is significant if it is sufficiently different from how the expectation value differs from what is required from Ramsey theory. At a significance level of 0.01,the non-significant deviations are \underline{underlined}\footnote{The significance level selected depends on the cost function of the particular model. For instance, the significance level would scale with the cost associated with being wrong.} . The furthest deviation can be attributed to the Republican congressional voters and is an indication that a strong bias exists in their voting records. 


\subsection{Collaboration model}
\label{subsec:collab}

Suppose we have a collection of people $V$, working together on a communal project. 

As an example we look at economic trading data \cite{tradingdata}. Every country is represented by a node, and we add a blue edge from a country to its 5 largest importers and exporters by volume.. In this way, two countries are connected by a blue edge if their countries are historically economically connected and by a red edge if they are smaller trading partners. There is an asymmetry in the way edges are added, as for example, China only adds at most 10 blue edges to other countries, but many countries add blue edges to China. In this way it is possible for a country to have blue degree much higher than 10. This graph is best described as an Interaction Graph similar to the ``friends at a party''.

\begin{figure}[H]
\centering
\label{fig:country_circle}
\includegraphics[scale=0.25]{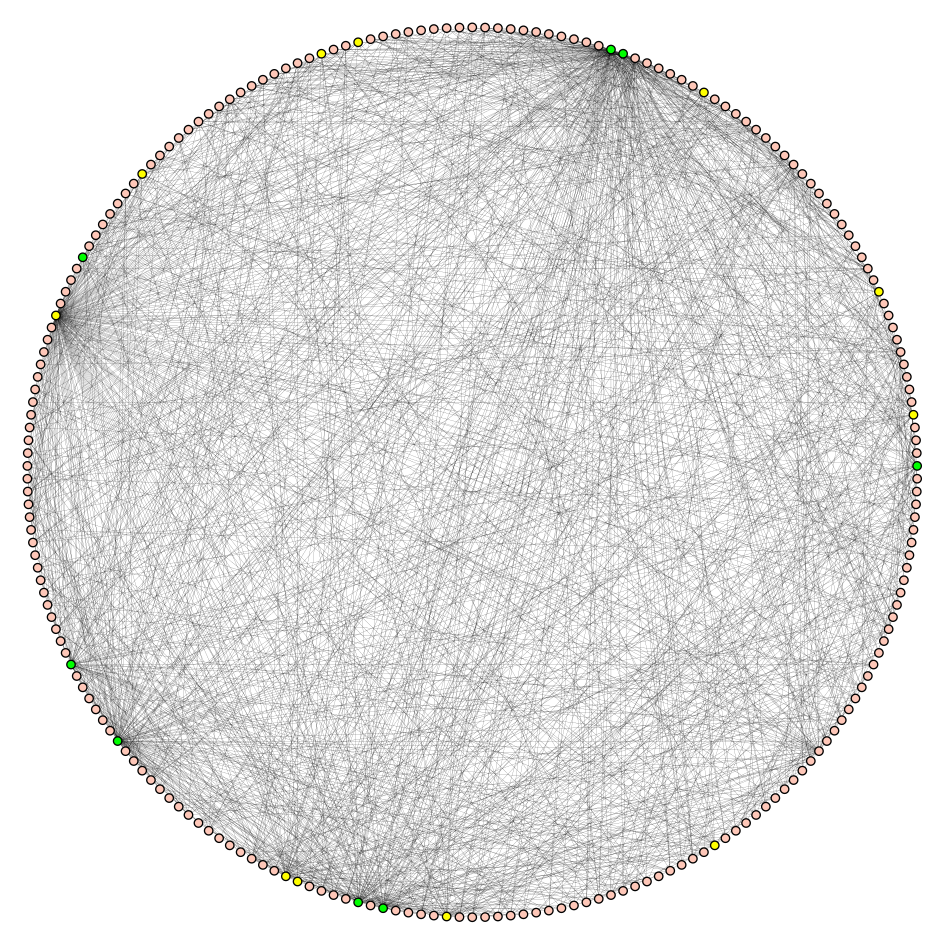}
\caption{Countries are arranged alphabetically starting at the top and going counterclockwise. The green nodes are the G7 and G20 countries. The graph has 214 vertices, 1363 blue edges, the average blue degree is 12.7, the five highest blue degrees are 162 (China), 125 (United States), 96 (Germany), 66 (France) and Italy (61). The largest complete subgraph has 8 vertices: Algeria, China, France, Germany, Italy, Spain, United Kingdom, United States, forming a $K_8$. The largest independent set has 70 vertices, forming an $I_{70}$.}
\end{figure}

For $N=214$ countries, the number of monochromatic triangles equals $85.0 \%$ of the total number $\binom{N}{3}$ of triangles in a $K_{214}$. These monochromatic triangles are almost entirely red, representing a lack of strong trade relations. This is significantly more than the required number of triangles given by Goodman's formula, which at $N=214$ is 24.7 \%.

\begin{table}[H]
\centering
\caption{}
\begin{tabular}{lr}
Percentage of red $K_3$ & 84.8 \% \\
Percentage of blue $K_3$ & 0.2 \% \\
Percentage of mono-chromatic triangles & 85.0 \% \\
Goodman-type lower bound & 24.7 \% \\
\end{tabular}
\end{table}

Since this graph has a threshold of only the top 5 trading partners for each country, it can be seen as a discrete sample of the threshold graph that would exist on the scale [$t_{min} = $ top trading partner to $t_{max} = $ all trading partners]. In order to determine if the percentage of monochromatic triangles in this graph can be interpreted as meaningful evidence that the global economy connected with a strong dichotomy, we need to measure its p-value. For $n=214$ countries, a threshold of $t=5$ corresponds to $t=0.0234$ on a normalized scale of $[0,1]$. When $t = 0.0234$, the expected deviation for the total number of monochromatic triangles from those required by Ramsey theory has a $\chi^2=4.236$, whereas the trading graph has a $\chi^2 = 2.907$. The difference between these is $ 1.329$, which corresponds to a p-value of 0.248983, which is not statistically significant. We can therefore not reject the null-hypothesis that this trade graph is random.

While we cannot reject $H_0$ based on the number of superfluous monochromatic $K_3$'s in the trading data, the presence of higher dimensional complete subgraphs might provide sufficient evidence.

We can compute the percentage of monochromatic $K_4$, and the percentage of monochromatic $K_5$. This is computationally complex, so we computed these percentages for only small $N$. 

\begin{table}[H]
\centering
\caption{Data for the country graph in section \ref{subsec:collab}.}
\begin{tabular}{lcc}
N & Percentage of mono-chromatic $K_N$ (\%) & Goodman-type lower bound\\ \hline
 3 & 85 & 25 \\
 4 & 74 & 3  \\
 5 & 62 & $<1$ \\
\end{tabular}
\end{table}
It is natural to then ask what happens when we consider larger substructures, that is $K_4, K_5,...,K_{N}$ instead of triangles. 
\subsubsection{$\chi^2$ for Higher Dimensions}

For higher dimensions, there is no analogue of Goodman's formula, which we would expect to give us a percentage of $\frac{1}{32}$ for $K_4$, $\frac{1}{16384}$ for $K_5$, etc... using the same methods described in Corollary 9. In \cite{thomason1989}, Thomason has shown that an upper bound for the corresponding percentage of monochromatic $K_4$ is $\frac{1}{33}$, although it is not known if this is tight. In the same work he gave an upper bound on the number of monochromatic $K_m$, as $0.936 \cdot 2^{1-\binom{m}{2}}$.

\begin{figure}[H]
\centering
\label{fig:expected_value_4_5}
\includegraphics[scale=0.2]{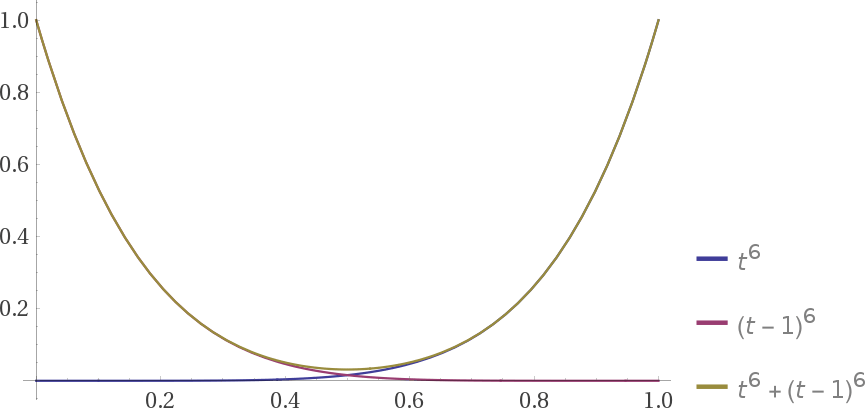}\includegraphics[scale=0.2]{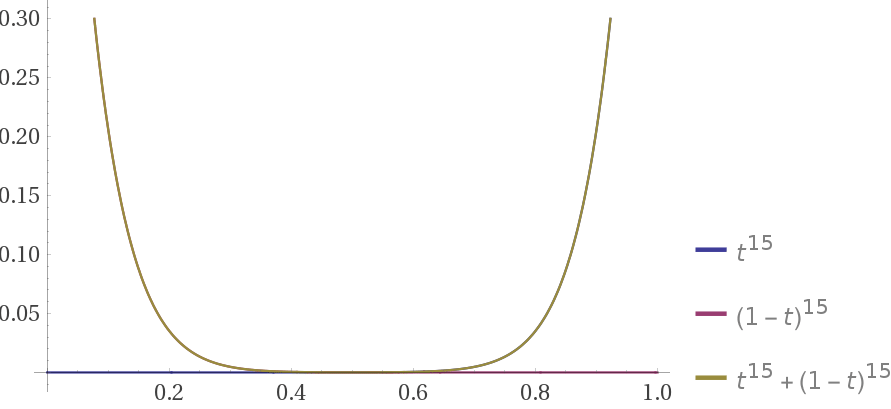}
\caption{The expected number of monochromatic $K_4$ and $K_5$s as a function of $t$. The Goodman-type upper bound for $K_4$ is 0.0295, and 0.00183 for $K_5$. }
\end{figure}

For the $\chi^2$'s related to larger substructures, Thomason's upper bound can be used in the same way that Goodman's is used for $K_3$, with the understanding that this will give us an upper bound on a graph's deviation from what's required by Ramsey theory. Our new $\bar{\chi}^2$ is an average of each $K_m$'s associated $\chi^2$ and can include up to $N-$dimensional substructures:

$$ \bar{\chi}^2 = \frac{1}{N-3}\sum_{i=3}^{N} \chi^2_{K_i} $$

If instead we increase the number of colors and therefore allow for more than two classifications, a perfect answer for three colors and triangles is given by \cite{cummings2013}. 

\section{Applications to transitivity}
\label{sec:trans}

When we have sufficient evidence to reject $H_0$, we define a non-random graph in terms of its transitivity. Transitivity can be thought of as the likelihood that a relationship in a dataset is meaningful and therefore not spurious. Let's again consider the model for the party problem: the nodes are people at a party and we assign a blue (B) edge between two people if they are friends (and red (R) if they are not friends).

\begin{definition}
 Transitivity: A binary relation $R$ on $D$ is transitive if  $\forall v_i,v_j,v_k \in D$ if $v_i R v_j$ and $v_j R v_k$ then $v_i R v_k$. 
\end{definition}
In this setting, we first remark that the blue ``friend" relation is not by-default transitive, and neither is the red ``not friend" relation. For example, I am friends with someone who does not know my brother.

It is easy to see that the only way for the red relation to be transitive is if \textit{all} edges are red in a particular subgraph. Similarly, the blue relation is transitive only if all edges are blue. Typically, such a graph will not be transitive in both relations.

Transitivity can be described in terms of monochromatic triangles, specifically three vertices $v_i,v_j,v_k$ are members of a graph that is not transitive when the edges between them are not monochromatic. In this way, the percentage of monochromatic triangles in a graph is a measure of how transitive a graph is. In the context of uncolored graphs this has been studied as the clustering coefficient. However, by looking at two colored graphs, Goodman's formula implies that there is a lower limit on how non-transitive a graph. We know that least 0.25 of its triangles must be monochromatic in the case of a 2 colored graph. The higher the observed percentage is than 0.25, the more transitive the graph is, and this can be measured in terms of $\chi^2$.

Let's use this to interpret the results from section \ref{subec:voting}. Suppose we have three democrats $v_i,v_j,v_k$ and we know that $v_i R v_j$ iff $v_j R v_k$; that is, the relationship between $v_i$ and $v_j$ is the exact same as the one between $v_j$ and $v_k$ (although we don't necessarily know if both have an edge or not).

We ask: how likely is it that the relationship between $v_i$ and $v_j$ is the same as the one between $v_i$ and $v_k$, i.e. that the triangle is transitive?

\begin{theorem} Let $G$ be a complete graph whose edges are colored red (R) or blue.  The percentage of monochromatic paths of length 2 that complete to a monochromatic triangle is measured by
\[
	\frac{3f(G)}{\binom{N}{3} + 2 f(G)},
\]
where $f(G)$ is the number of monochromatic triangles in $G$.
\end{theorem}

\begin{proof} This quantity comes from the observation that every monochromatic triangle contains three monochromatic paths of length 2, but each non-monochromatic triangle contains precisely one monochromatic path of length 2. For ease of computation we use that (the number of non-monochromatic triangles) + $3\times$(the number of monochromatic triangles) is $(\binom{N}{3}-f(G))+(3 f(G))=\binom{N}{3} + 2 f(G)$, since $\binom{n}{3}$ is the total number of triangles. Thus, $\binom{N}{3} + 2 f(G)$ is the total number of monochromatic paths of length 2 in $G$, since this counts every non-monochromatic triangle once and counts every monochromatic triangle three times.
\end{proof}

By using Goodman's formula, this observation above translates to the following (completely expected) result:

\begin{proposition}\label{prop:goodman_trans_percent} Let $G$ be a graph with $N$ vertices and edge-colored with red and blue. The ratio of monochromatic paths in $G$ that are part of a monochromatic triangle is asymptotically at least $0.5$.
\end{proposition}

The observation above provides an efficient way to compute the ratio of monochromatic paths in $G$ that are part of a monochromatic triangle. We, for example, don't need to count the number of monochromatic paths directly.

\subsection{Application to previous examples}

\subsubsection{Application to voting records}
In the case of the threshold graphs from section \ref{subec:voting}, the threshold graph $G(t)$ with the minimum ``transitivity'' percentage is precisely the threshold graph with the minimum number of monochromatic triangles, namely $t=9$ ($52.7 \%$). Analogously, for $D(t)$ this occurs at $t=7$ ($66.6 \%$) and for $R(t)$ this occurs at $t=5$ ($72.0 \%$).

\begin{table}[H]
\centering
\caption{Transitivity numbers for the threshold graphs. The minimum values are boxed.}
\begin{tabular}{llll}
t & $G(t)$ & $D(t)$ & $R(t)$ \\ \hline
0 & 1.000 & 1.000 & 1.000 \\
1 & 0.997 & 0.997 & 0.990 \\
2 & 0.984 & 0.983 & 0.935 \\
3 & 0.947 & 0.944 & 0.819 \\
4 & 0.888 & 0.874 & 0.730 \\
5 & 0.811 & 0.785 & \fbox{0.719} \\
6 & 0.719 & 0.701 & 0.754 \\
7 & 0.626 & \fbox{0.665} & 0.806 \\
8 & 0.552 & 0.687 & 0.859 \\
9 & \fbox{0.526} & 0.747 & 0.909 \\
10 & 0.561 & 0.809 & 0.941 \\
11 & 0.637 & 0.868 & 0.960 \\
12 & 0.727 & 0.915 & 0.976 \\
13 & 0.816 & 0.952 & 0.984 \\
14 & 0.896 & 0.980 & 0.987 \\
15 & 0.959 & 0.993 & 0.992 \\
16 & 0.989 & 0.998 & 0.994 \\
17 & 1.000 & 1.000 & 1.000
\end{tabular}
\end{table}
\begin{center}
\centering
\begin{tikzpicture}[scale=.75]
\begin{axis}[legend style={at={(1,1)},anchor=south east}]
\addplot[black,mark=otimes] table [x=a, y=e, col sep=comma] {data.csv};
\addplot[blue,mark=diamond] table [x=a, y=f, col sep=comma] {data.csv};
\addplot[red,mark=triangle] table [x=a, y=g, col sep=comma] {data.csv};
\addplot[red,sharp plot,update limits=false] coordinates {(-3,0.50) (20,0.50)};
\legend{All Monochromatic Triangles in G( t ),Democrats D( t ),Republicans R( t ),Goodman}
\end{axis}
\end{tikzpicture}
\end{center}

\subsubsection{Application to global trading data}

China provides an interesting example of a country that is part of many non-monochromatic triangles, because China has an edge to $162$ of the $213$ other countries, and of those 162 countries only $6.9 \%$ of edges are present among China's neighbors. So only a small percentage of China's neighbors are themselves directly connected. This contributes to slightly lowering the percentage of transitivity in the larger graph.

In total, using all countries, $94.4 \%$ of all monochromatic paths complete to an edge of the same color. This is well above the $50 \%$ guaranteed by Proposition \ref{prop:goodman_trans_percent}. Again, a complication is introduced by only looking at one threshold level rather than calculating the entire $\chi^2$.

\section{Conclusions and questions}
\label{sec:concl}

We now make two major calls to use these methods: applications and development of related theory.

\subsection{Theory building}

This use of Goodman's formula suggests the need for other quantitative Ramsey statements. For higher dimensional objects, we mention a couple that already exist and some that have yet to be developed.

A recent survey of Ramsey bounds for hypergraphs is a useful place to see the current best known bounds for various Ramsey numbers \cite{mubayi2017survey}. This survey also goes through proof sketches, many of which contain a weak Goodman-style lower bound. These bounds typically come from a use of the probabilistic method (see for example \cite{alon2004probabilistic}).

In general, the probabilistic bounds provide a first non-trivial upper bound on the percentage of monochromatic structures, and improving them can be difficult. In order to use Ramsey theory in a generalized way, a closed form analogous to Goodman's formula needs to be developed for all $K_n$ subgraphs and all $C_n$-colored graphs. 

\subsection{Further applications}

The case of triangles is simple, but still captures the quantitative notion of transitivity of a relation. Additionally, counting the number of monochromatic triangles in a graph is computationally efficient.

Further progress could be motivated by finding interpretations for other quantitative Ramsey statements. For example, a quantitative version of Van der Waerden's theorem for a fixed length. That is, given a 2-coloring of the points $\{1,2,\ldots, 9\}$ it is known that there must be at least one arithmetic progression of length $3$ (i.e. $a_0, a_0 + m, a_0 + 2m$) where all points are the same color. The following question has a reasonable answer in \cite{sjoland2014enumeration}, which has serious mathematical content:

\begin{question} For $N$ sufficiently large. Give reasonable lower-bounds and upper bounds on the percentage of monochromatic 3-term progressions that must exist for any $2$-coloring of $\{0, 1,2, \ldots, n\}$.
\end{question}

In \cite{sjoland2014enumeration}, it is shown that asymptotically, at least $25 \%$ of all 3-term such arithmetic progressions must be monochromatic. This extended results of \cite{cameron2007monochromatic}. In their setting arithmetic progressions are allowed to ``wrap around". That is, in $\{0, 1,2,3,4,5,6,7\}$ the triple $\{5,7,1\}$ is considered a 3-term progression.

For $4$-term progressions, see \cite{wolf2010minimum} and the strengthening \cite{lu2012monochromatic}. Both of these are non-trivial results.

The next step is to interpret $3$-term progressions (or $4$-term progressions) in a data-set in a meaningful, physical way.

\subsection{Closing remarks}

We believe that the connections between data science and Ramsey theory are still largely unmade and will prove to be profound. We have shown that Ramsey theory can be used to rigorously define spurious correlations in datasets, and how deviations from the number of required spurious correlations might be meaningful in terms of transitivity. 

On behalf of all authors, the corresponding author states that there is no conflict of interest.

\bibliographystyle{plain}
\bibliography{references}

\end{document}